\documentclass[a4paper,10pt]{amsart}
\usepackage[utf8x]{inputenc}
\usepackage{amsmath}
\usepackage{amsthm}
\usepackage{amssymb}
\usepackage{hyperref}
\newtheorem{lemma}{Lemma}
\newtheorem{theorem}{Theorem}

\newtheorem{corollary}{Corollary}
\newtheorem*{theorem*}{Theorem}
\newtheorem*{corollary*}{Corollary}
\begin{document}

\title{On the $H$-triangle of generalised nonnesting partitions}
\author{Marko Thiel}
\address{Department of Mathematics, University of Vienna, Oskar-Morgenstern-Platz 1, 1090 Vienna, Austria}

\begin{abstract} With a crystallographic root system $\Phi$ and a positive integer $k$, there are associated two Fu{\ss}-Catalan objects, the set of $k$-generalised nonnesting partitions $NN^{(k)}(\Phi)$, and the generalised cluster complex $\Delta^{(k)}(\Phi)$.
These possess a number of enumerative coincidences, many of which are captured in a surprising identity, first conjectured by Chapoton for $k=1$ and later generalised to $k>1$ by Armstrong. We prove this conjecture, obtaining some structural and enumerative results on $NN^{(k)}(\Phi)$ along the way, including an earlier conjecture by Fomin and Reading giving a refined enumeration by Fu{\ss}-Narayana numbers.
\end{abstract}

\maketitle
\section{Introduction}
For a crystallographic root system $\Phi$, there are three well-known Coxeter-Catalan objects \cite{armstrong09generalized}: the set of noncrossing partitions $NC(\Phi)$, the set of nonnesting partitions $NN(\Phi)$ and the cluster complex $\Delta(\Phi)$.
The former two and the set of facets of the latter are all counted by the same numbers, the Coxeter-Catalan numbers $Cat(\Phi)$. For the root system of type $A_{n-1}$, these reduce to objects counted by the classical Catalan numbers $C_n=\frac{1}{n+1}\binom{2n}{n}$, namely the set of noncrossing partitions of $[n]=\{1,2,\ldots,n\}$, the set of nonnesting partitions of $[n]$ and the set of triangulations of a convex $(n+2)$-gon, respectively.\\
\\
Each of these Coxeter-Catalan objects has a generalisation \cite{armstrong09generalized}, a Fu{\ss}-Catalan object defined for each positive integer $k$. These are the set of $k$-divisible noncrossing partitions $NC_{(k)}(\Phi)$, the set of $k$-generalised nonnesting partitions $NN^{(k)}(\Phi)$ and the generalised cluster complex $\Delta^{(k)}(\Phi)$. They specialise to the corresponding Coxeter-Catalan objects when $k=1$. The former two and the set of facets of the latter are counted by Fu{\ss}-Catalan numbers $Cat^{(k)}(\Phi)$, which specialise to the classical Fu{\ss}-Catalan numbers $C^{(k)}_n=\frac{1}{kn+1}\binom{(k+1)n}{n}$ in type~$A_{n-1}$.\\
\\
The enumerative coincidences do not end here. Chapoton defined the $M$-triangle, the $H$-triangle and the $F$-triangle, which are polynomials in two variables that encode refined enumerative information on $NC(\Phi)$, $NN(\Phi)$ and $\Delta(\Phi)$ respectively \cite{chapoton04enumerative,chapoton06sur}. This allowed him to formulate the $M=F$ conjecture \mbox{\cite[Conjecture 1]{chapoton04enumerative}} and the $H=F$ conjecture \cite[Conjecture 6.1]{chapoton06sur} relating these polynomials through invertible transformations of variables. These conjectures were later generalised to the corresponding Fu{\ss}-Catalan objects by Armstrong \cite[Conjecture 5.3.2.]{armstrong09generalized}.
The $M=F$ conjecture was first proven by Athanasiadis \cite{athanasiadis07associahedra} for $k=1$, and later by Krattenthaler \cite{krattenthaler06ftriangle,krattenthaler06mtriangle} and Tzanaki \cite{tzanaki08faces} for $k>1$. We establish the $H=F$ conjecture in the generalised form due to Armstrong, obtaining some structural and enumerative results on $NN^{(k)}(\Phi)$ along the way, including a refined enumeration by Fu{\ss}-Narayana numbers conjectured by Fomin and Reading \cite[Conjecture 10.1.]{fomin05generalized}.
\section{Definitions and the Main Result}
Let $\Phi=\Phi(S)$ be a crystallographic root system with a simple system $S$. Then $\Phi=\Phi^+\sqcup-\Phi^+$ is the disjoint union of the set of positive roots $\Phi^+$ and the set of negative roots $-\Phi^+$. Every positive root can be written uniquely as a linear combination of the simple roots and all coefficients of this linear combination are nonnegative integers. For further background on root systems, see \cite{humphreys90reflection}.
Define the \textit{root order} on $\Phi^+$ by 
$$\beta\geq\alpha\text{ if and only if }\beta-\alpha\in\langle S\rangle_{\mathbb{N}}\text{,}$$
that is, $\beta\geq\alpha$ if and only if $\beta-\alpha$ can be written as a linear combination of simple roots with nonnegative integer coefficients.
The set of positive roots $\Phi^+$ with this partial order is called the \textit{root poset}. An \textit{order filter} in this poset is a subset $I$ of $\Phi^+$ such that whenever $\alpha\in I$ and $\beta\geq\alpha$, then also $\beta\in I$. The \textit{support} of a root $\beta\in\Phi^+$ is the set of all $\alpha\in S$ with $\alpha\leq\beta$.
For $\alpha\in S$, define $I(\alpha)$ as the order filter generated by $\alpha$, that is the set of all roots $\beta\in\Phi^+$ that have $\alpha$ in their support.
We call a root system \textit{classical} if all its irreducible components are of type $A$, $B$, $C$ or $D$. Otherwise we call it \textit{exceptional}.\\
\\
Let $V=\langle\Phi\rangle=\langle S\rangle$ be the ambient vector space spanned by the roots. Define the \textit{$k$-Catalan arrangement} as the hyperplane arrangement in $V$ given by the hyperplanes $H_{\alpha,i}=\{x\in V\mid \langle x,\alpha\rangle=i\}$ for $\alpha\in\Phi$ and $i\in\{0,1,\ldots,k\}$.
The connected components of $V\backslash\bigcup_{\alpha,i}H_{\alpha,i}$ are called the \textit{regions} of the arrangement, and those regions $R$ with $\langle x,\alpha\rangle\geq0$ for all $x\in R$ and all $\alpha\in\Phi^+$ are called \textit{dominant}.
A supporting hyperplane of a dominant region $R$ is called a \textit{wall} of $R$, it is a \textit{separating wall} or \textit{floor} of $R$ if the origin and $R$ lie on different sides of the hyperplane.
It is a \textit{ceiling} if $R$ and the origin lie on the same side of the hyperplane. A hyperplane is called \textit{$i$-coloured} if it is of the form $H_{\alpha,i}$ for some $\alpha\in\Phi^+$.\\
\\
Now let $k$ be a positive integer and consider a descending (multi)chain of order filters: a tuple $I=(I_1,I_2,\ldots,I_k)$ with $I_1\supseteq I_2\supseteq\ldots\supseteq I_k$.
This chain is called \textit{geometric} if the following conditions hold:
$$(I_i+I_j)\cap\Phi^+\subseteq I_{i+j}\text{ for all }i,j\in\{0,1,\ldots,k\}\text{ and }$$
$$(J_i+J_j)\cap\Phi^+\subseteq J_{i+j}\text{ for all }i,j\in\{0,1,\ldots,k\}\text{ with }i+j\leq k\text{,}$$
where $I_0=\Phi^+$, $J_i=\Phi^+\backslash I_i$ for $i\in\{0,1,\ldots,k\}$, and $I_i=I_k$ for $i>k$. 
The set $NN^{(k)}(\Phi)$ of \textit{$k$-generalised nonnesting partitions} of $\Phi$ is the set of all geometric chains of $k$ order filters in the root poset of $\Phi$.\\
\\
For $I\in NN^{(k)}(\Phi)$, and $\alpha\in\Phi^+$, define $$k_{\alpha}(I)=\text{max}\{k_1+k_2+\ldots+k_r\mid\alpha=\alpha_1+\alpha_2+\ldots+\alpha_r\text{ and }\alpha_i\in I_{k_i}\text{ for all }i\}\text{.}$$ A positive root $\alpha$ is called a \textit{rank $l$ indecomposable element} \cite[Definition 3.8.]{athanasiadis05refinement} of a $k$-nonnesting partition $I$ if $\alpha\in I_l$, $k_{\alpha}(I)=l$, $\alpha\notin I_i+I_j$ for all $i,j\in\{0,1,\ldots,k\}$ with $i+j=l$, and for all $\beta\in\Phi^+$, if $\alpha+\beta\in I_t$ for some $t\leq m$ and $k_{\alpha+\beta}(I)=t$, then $\beta\in I_{t-l}$.
A rank $l$ indecomposable element $\alpha$ is called a \textit{rank $l$ simple element} of $I$ if $\alpha\in S$. The rank $k$ indecomposable elements of $I$ are precisely the elements of $I_k$ that are not in $I_i+I_j$ for all $i,j\in\{0,1,\ldots,k\}$ with $i+j=k$ \cite[Lemma 3.9.]{athanasiadis05refinement}. Notice that all $\alpha\in I_k\cap S$ are automatically rank $k$ indecomposable.\\
\\
Now we can define the \textit{$H$-triangle} \cite[Definition 5.3.1.]{armstrong09generalized} as
$$H^k_{\Phi}(x,y)=\sum_{I\in NN^{(k)}(\Phi)}x^{|i(I)|}y^{|s(I)|}\text{,}$$
where $i(I)$ is the set of rank $k$ indecomposable elements of $I$ and $s(I)$ is the set of rank $k$ simple elements of $I$.\\
\\
There is a natural bijection $\phi$ \cite[Theorem 3.6.]{athanasiadis05refinement} from the set of dominant regions $R$ of the $k$-Catalan arrangement of $\Phi$ to $NN^{(k)}(\Phi)$, such that $H_{\alpha,i}$ is a floor of $R$ if and only if $\alpha$ is a rank $i$ indecomposable element of $\phi(R)$.\\

Let $\Phi^{(k)}_{\geq-1}$ be the set of \textit{k-coloured almost positive roots} of $\Phi$, containing one uncoloured copy of each negative simple root and $k$ copies of each positive root, each with a different colour from the colour set $\{1,2,\ldots,k\}$. Then there exists a symmetric binary relation called \textit{compatibility} \cite[Definition 3.1.]{fomin05generalized} on $\Phi^{(k)}_{\geq-1}$ such that all uncoloured negative simple roots are pairwise compatible and for $\alpha\in S$ an uncoloured simple root and $\beta^{(i)}\in\Phi^+$ a positive root with colour $i$, $-\alpha$ is compatible with $\beta^{(i)}$ if and only if $\alpha$ is not in the support of $\beta$.
Notice that the colour $i$ of $\beta^{(i)}$ does not matter in this case.\\
\\
Define a simplicial complex $\Delta^{(k)}(\Phi)$ as the set of all subsets $A\subseteq\Phi^{(k)}_{\geq-1}$ such that all $k$-coloured almost positive roots in $A$ are pairwise compatible. This is the \textit{generalised cluster complex} of $\Phi$.
This simplicial complex is \textit{pure}, all facets have cardinality $n$, where $n=|S|$ is the rank of $\Phi$. The subcomplex consisting of those faces containing no negative simple root is called the \textit{positive part} of $\Delta^{(k)}(\Phi)$.\\
\\
Now we can define the \textit{$F$-triangle} \cite[Definition 5.3.1.]{armstrong09generalized} as
$$F^k_{\Phi}(x,y)=\sum_{A\in\Delta^{(k)}(\Phi)}x^{|A^+|}y^{|A^-|}=\sum_{l,m}f_{l,m}^k(\Phi)x^ly^m\text{,}$$
where $A^+$ is the set of coloured positive roots in $A$ and $A^-$ is the set of uncoloured negative simple roots in $A$. Thus $f_{l,m}^k(\Phi)$ is the number of faces of $\Delta^{(k)}(\Phi)$ containing exactly $l$ coloured positive roots and exactly $m$ uncoloured negative simple roots.\\
\\
Consider the Weyl group $W=W(\Phi)$ of the root system $\Phi$. A \textit{Coxeter element} in this group is a product of all the simple reflections in some order. Let $T$ denote the set of reflections in $W$.
For $w\in W$, define the \textit{absolute length} $l_T(w)$ of $w$ as the minimal $l$ such that $w=t_1t_2\cdots t_l$ for some $t_1,t_2,\ldots, t_l\in T$.
Define the \textit{absolute order} on $W$ by 
$$u\leq_T v \text{ if and only if } l_T(u)+l_T(u^{-1}v)=l_T(v)\text{.}$$
Fix a Coxeter element $c\in W$. A \textit{$k$-delta sequence} is a sequence $\delta=(\delta_0,\delta_1,\ldots,\delta_k)$ with $\delta_i\in W$ for all $i\in\{0,1,\ldots,k\}$ such that $c=\delta_0\delta_1\cdots\delta_k$ and $l_T(c)=\sum_{i=0}^kl_T(\delta_i)$.
Define a partial order on $k$-delta sequences by 
$$\delta\leq\epsilon\text{ if and only if }\delta_i\leq_T\epsilon_i\text{ for all }i\in\{1,2,\ldots,k\}\text{.}$$
The set of $k$-delta sequences with this partial order is called the poset of \textit{\nolinebreak[4]{$k$-divisible} noncrossing partitions} $NC_{(k)}(\Phi)$ \cite[Definition 3.3.1.]{armstrong09generalized}. We drop the choice of the Coxeter element $c$ from the notation, since a different choice of Coxeter element results in a different but isomorphic poset.\\
\\
Now we can define the \textit{$M$-triangle} \cite[Definition 5.3.1]{armstrong09generalized} as
$$M^k_{\Phi}(x,y)=\sum_{\delta,\epsilon\in NC_{(k)}(\Phi)}\mu(\delta,\epsilon)x^{n-rk(\epsilon)}y^{n-rk(\delta)}\text{,}$$
where $rk$ is the rank function of the graded poset $NC_{(k)}(\Phi)$, $\mu$ is its Möbius function, and $n$ is the rank of $\Phi$.\\
\\
As mentioned in the introduction, $NC_{(k)}(\Phi)$, $NN^{(k)}(\Phi)$ and the set of facets of $\Delta^{(k)}(\Phi)$ are all counted by the same number $Cat^{(k)}(\Phi)$.
But more is true: define the \textit{Fu{\ss}-Narayana number} $Nar^{(k)}(\Phi,i)$ as the number of elements of $NC_{(k)}(\Phi)$ of rank $n-i$ \cite[Definition 3.5.4.]{armstrong09generalized}.
Let $(h_0,h_1,\ldots,h_n)$ be the \textit{$h$-vector} of $\Delta^{(k)}(\Phi)$, defined by the relation 
$$\sum_{i=0}^nh_ix^{n-i}=\sum_{l,m}f_{l,m}(x-1)^{n-(l+m)}\text{.}$$
Then $h_{n-i}=Nar^{(k)}(\Phi,i)$ for all $i\in\{0,1,\ldots,n\}$ \cite[Theorem 3.2.]{fomin05generalized}. In this paper, we prove the following conjecture of Fomin and Reading that relates the Fu{\ss}-Narayana numbers to $k$-generalised nonnesting partitions as well.
\begin{theorem}[\protect{\cite[Conjecture 10.1.]{fomin05generalized}}]\label{Nar}
 For a crystallographic root system $\Phi$, the number of $k$-generalised nonnesting partitions of $\Phi$ that have exactly $i$ indecomposable elements of rank $k$ equals the Fu{\ss}-Narayana number $Nar^{(k)}(\Phi,i)$.
\end{theorem}

The main result of this paper is the following theorem, conjectured by Armstrong.
\begin{theorem}[\protect{\cite[Conjecture 5.3.2.]{armstrong09generalized}}]\label{H=F}
If $\Phi$ is a crystallographic root system of rank $n$, then
$$H^k_{\Phi}(x,y)=(x-1)^nF^k_{\Phi}\left(\frac{1}{x-1},\frac{1+(y-1)x}{x-1}\right)\text{.}$$
\end{theorem}

We first establish \textbf{Theorem \ref{Nar}}. Then we find a combinatorial bijection for $k$-generalised nonnesting partitions that leads to a differential equation for the $H$-triangle analogous to one known for the $F$-triangle.
Using both of these differential equations, \textbf{Theorem \ref{Nar}}, and induction on the rank $n$, we prove \textbf{Theorem \ref{H=F}} by showing that the derivatives with respect to $y$ of both sides of the equation as well as their specialisations at $y=1$ agree.\\
\\
After proving \textbf{Theorem \ref{H=F}}, we use it to deduce various corollaries. Using the $M=F$ (ex-)conjecture, we also relate the $H$-triangle to the $M$-triangle. This allows us to transfer a remarkable instance of combinatorial reciprocity observed by Krattenthaler \cite[Theorem 8]{krattenthaler06mtriangle} for the $M$-triangle to the $H$-triangle.
From \textbf{Theorem \ref{Nar}} we also get a proof of a conjecture of Athanasiadis and Tzanaki \cite[Conjecture 1.2.]{athanasiadis06cluster} that relates the $h$-vector of the positive part of $\Delta^{(k)}(\Phi)$ to the enumeration of bounded dominant regions in the $k$-Catalan arrangement of $\Phi$ by their number of $k$-coloured ceilings.
%
\section{Proof of Theorem \ref{Nar}}
Let $N^{(k)}(\Phi,i)$ be the number of $k$-generalised nonnesting partitions of $\Phi$ that have exactly $i$ indecomposable elements of rank $k$. If $\Phi$ is a classical root system, $N^{(k)}(\Phi,i)$ is known \cite[Proposition 5.1.]{athanasiadis05refinement} to equal the Fu{\ss}-Narayana number $Nar^{(k)}(\Phi,i)$. We wish to verify this also when $\Phi$ is of exceptional type. It suffices to do this for $\Phi$ irreducible. Let $n$ be the rank, $h$ the Coxeter number and $\tilde{\alpha}$ the highest root of $\Phi$. Let $L(\Phi^{\vee})$ be the coroot lattice of $\Phi$. Define $\overline{A_\circ}=\{x\in V\mid\langle \alpha_i,x\rangle\geq0\text{ for all }\alpha_i \in S, \langle \tilde{\alpha},x\rangle\leq 1\}$, the \textit{fundamental simplex} of $\Phi$. 
In order to prove \textbf{Theorem \ref{Nar}}, we use the following result, due to Athanasiadis.
\begin{theorem}[\protect{\cite[Corollary 4.4.]{athanasiadis05refinement}}]\label{ath}
 The number of $k$-generalised nonnesting partitions of $\Phi$ that have exactly $i$ indecomposable elements of rank $k$ equals the number of elements of the coroot lattice in the $(kh+1)$-fold dilation of the fundamental simplex that are contained in exactly $i$ of its walls.
\end{theorem}

\begin{proof}[Proof of Theorem \ref{Nar}]
We wish to count elements of $L(\Phi^{\vee})\cap t\overline{A_\circ}$ by the number of walls they are contained in and then specialise to $t=kh+1$. Now the simple coroots $\{\alpha_1^{\vee},\ldots,\alpha_n^{\vee}\}$ are a $\mathbb{Z}$-basis of $L(\Phi^{\vee})$. So if $\{e_1,\ldots,e_n\}$ is the standard basis of $\mathbb{R}^n$, the linear map $\Gamma$ defined by $\Gamma(\alpha_i^{\vee})=e_i$ gives a lattice isomorphism from $L(\Phi^{\vee})$ to $\mathbb{Z}^n$.
Let $P=\Gamma(\overline{A_\circ})$. Then 
\begin{align*}
 P&=\{y\in \mathbb{Z}^n\mid\left\langle\alpha_i,\Gamma^{-1}(y)\right\rangle\geq0\text{ for all }\alpha_i\in S, \left\langle\tilde{\alpha},\Gamma^{-1}(y)\right\rangle\leq 1\} \\
 &=\{y\in \mathbb{Z}^n\mid\left\langle\alpha_i,\sum_{j=1}^ny_j\alpha_j^{\vee}\right\rangle\geq0\text{ for all }\alpha_i\in S, \left\langle\sum_{i=1}^nc_i\alpha_i,\sum_{j=1}^ny_j\alpha_j^{\vee}\right\rangle\leq 1\} \\
 &=\{y\in \mathbb{Z}^n\mid\sum_{j=1}^na_{ji}y_j\geq0\text{ for all }\alpha_i\in S, \sum_{i=1}^nc_i\sum_{j=1}^na_{ji}y_j\leq 1\} \\
 &=\{y\in \mathbb{Z}^n\mid A^{T}y\geq0\text{ for all }\alpha_i\in S, c^TA^Ty\leq 1\}\text{.}
\end{align*}
Here $A=(a_{ij})=(\left\langle\alpha_i^{\vee},\alpha_j\right\rangle)$ is the Cartan matrix of $\Phi$, and $c=(c_1,\ldots,c_n)$ is the vector of coefficients in the simple root expansion of the highest root $\tilde{\alpha}=\sum_{i=1}^nc_i\alpha_i$.
Now since $\Gamma$ is an isomorphism, it maps walls to walls, so instead of counting elements of $L(\Phi^{\vee})\cap t\overline{A_\circ}$ by the number of walls they are contained in we count elements of $\Gamma(L(\Phi^{\vee})\cap t\overline{A_\circ})=\mathbb{Z}^n\cap tP$ by the number of walls they are contained in.\\
\\
Let $B$ be a face of $P$, and let $f_B(t)=|\mathbb{Z}^n\cap tB|$ be the number of lattice points in $tB$. Let $g_B(t)$ be the number of lattice points in $tB$ that are in no other face of $P$ contained in $B$. Then by the Inclusion-Exclusion Principle, $$g_B(t)=\sum_{C\subseteq B}(-1)^{dim(B)-dim(C)}f_C(t)\text{.}$$
\\
Using the Macaulay2 interface of Normaliz \cite{bruns10normaliz}, for any irreducible exceptional root system $\Phi$, we calculate the Ehrhart series $F_B(z)=\sum_{t=0}^{\infty}f_B(t)z^t$ for all faces $B$ of $P$. This is a rational function in $z$, since the vertices of $P$ are in $\mathbb{Q}^n$. From this we get $$G_B(z)=\sum_{t=0}^{\infty}g_B(t)z^t=\sum_{C\subseteq B}(-1)^{dim(B)-dim(C)}F_C(z)\text{.}$$
Then we define $N_i(z)=\sum_BG_B(z)$, where the sum is over all faces $B$ of $P$ of dimension $n-i$. So the number of lattice points in $tP$ that are contained in exactly $i$ walls of $tP$ is equal to $[z^t]N_i(z)$, the coefficient of $z^t$ in the power series $N_i(z)$. So $N^{(k)}(\Phi,i)=[z^{kh+1}]N_i(z)$. Recall from Ehrhart theory that $[z^t]F_B(z)$, and therefore also $[z^t]G_B(z)$ and $[z^t]N_i(z)$, are quasipolynomials in $t$, with degree at most $n$ and period $p$, the least common multiple of the denominators of the vertices of $P$.
Thus $N^{(k)}(\Phi,i)=[z^{kh+1}]N_i(z)$ is a quasipolynomial in $k$, with period $p'=\frac{lcm(p,h)}{h}$. It turns out that $p'$ is either $1$ or $2$ in every case. Since $Nar^{(k)}(\Phi,i)$ is a polynomial in $k$ of degree at most $n$, in order to verify that $N^{(k)}(\Phi,i)=Nar^{(k)}(\Phi,i)$ for all $i$ and $k$, one only needs to check for every $i$ that they agree for the first $(n+1)p'$ values. Since the generating function $N_i(z)$ has been computed and explicit formulae for $Nar^{(k)}(\Phi,i)$ are known \cite[Theorem 3.5.6.]{armstrong09generalized}, this is easily accomplished.

\end{proof}

\section{The Bijection}
The other main tool in the proof of \textbf{Theorem \ref{H=F}} is the following bijection.
\begin{theorem}\label{Bij}
 For every simple root $\alpha\in S$, there exists a bijection $\Theta$ from the set of $k$-generalised nonnesting partitions $I\in NN^{(k)}(\Phi(S))$ with $\alpha\in I_k$ to $NN^{(k)}(\Phi(S\backslash\{\alpha\}))$.
 The rank $l$ indecomposable elements of $\Theta(I)$ are exactly the rank $l$ indecomposable elements of $I$ if $l<k$.
 The rank $k$ indecomposable elements of $\Theta(I)$ are exactly the rank $k$ indecomposable elements of $I$ except for $\alpha$.
 \end{theorem}
In order to prove this, we first need a basic lemma, implicit in \cite{athanasiadis05refinement}.
\begin{lemma}\label{ind}
 The rank $l$ indecomposable elements of a $k$-generalised nonnesting partitions $I\in NN^{(k)}(\Phi)$ are minimal elements of $I_l$.
 \begin{proof}
  Let $\alpha\in I_l$ be an indecomposable element. Suppose for contradiction that $\alpha$ is not minimal in $I_l$, say $\alpha>\beta\in I_l$.
  Then $\alpha=\beta+\sum_{i=1}^m\alpha_i$, where $\alpha_i\in S$ for all $i\in[m]$.
  So $\sum_{i=1}^m\alpha_i\in\Phi$ or $\beta+\sum_{i\neq j}\alpha_i\in\Phi$ for some $j\in[m]$, by \cite[Lemma 2.1. (i)]{athanasiadis05refinement}.
  In the first case, $\alpha=\beta+\sum_{i=1}^m\alpha_i$, with $\beta\in I_l$ and $\sum_{i=1}^m\alpha_i\in V_0$, so $\alpha$ is not indecomposable.
  In the second case, $\alpha=\beta+\sum_{i\neq j}\alpha_i+\alpha_j$, with $\beta+\sum_{i\neq j}\alpha_i\in I_l$ and $\alpha_j\in I_0$, so $\alpha$ is not indecomposable.
 \end{proof}
\end{lemma}
\begin{proof}[Proof of Theorem \ref{Bij}]
 Let $I(\alpha)$ be the order filter in the root poset generated by $\alpha$, that is the set of all positive roots that have $\alpha$ in their support.
 Define $\theta(I_i)=I_i\backslash I(\alpha)=I_i\cap\Phi(S\backslash\{\alpha\})$, where $\Phi(S\backslash\{\alpha\})$ is the root system with simple system $S\backslash\{\alpha\}$. Then let $\Theta(I)=(\theta(I_1),\theta(I_2),\ldots,\theta(I_k))$.\\
\\
  We claim that $\Theta(I)$ is a $k$-generalised nonnesting partition of $\Phi({S\backslash\{\alpha\}})$ and thus $\Theta$ is well-defined.\\
\\
    In order to see this, first observe that every $\theta(I_i)$ in $\Theta(I)$ is an order filter in the root poset of $\Phi(S\backslash\{\alpha\})$, and the $\theta(I_i)$ form a (multi)chain under inclusion. For all $i,j\in\{0,1,\ldots,k\}$,
    \begin{multline*}
     (\theta(I_i)+\theta(I_j))\cap\Phi^+(S\backslash\{\alpha\})\subseteq (I_i+I_j)\cap\Phi^+(S\backslash\{\alpha\})\\
     \subseteq I_{i+j}\cap\Phi^+(S\backslash\{\alpha\})=\theta(I_{i+j})\text{.}
    \end{multline*}
    Also $\theta(J_i)=J_i$ for all $i\in\{0,1,\ldots,k\}$, so $(\theta(J_i)+\theta(J_j))\cap\Phi^+(S\backslash\{\alpha\})\subseteq\theta(J_{i+j})$ for all $i,j$ with $i+j\leq k$.
    So $\Theta(I)$ is a geometric (multi)chain of order fiters in the root order of $\Phi({S\backslash\{\alpha\}})$, and the claim follows.
\\

 Now define a map $\Psi$ from $NN^{(k)}(\Phi(S\backslash\{\alpha\}))$ to the set of $I\in NN^{(k)}(\Phi(S))$ with $\alpha\in I_k$ by $\psi(I_i)=I_i\cup I(\alpha)$, and $\Psi(I)=(\psi(I_1),\psi(I_2),\ldots,\psi(I_k))$.\\
\\
 We claim that $\Psi(I)$ is a $k$-generalised nonnesting partition of $\Phi(S)$ and thus $\Psi$ is well-defined.\\
\\
   In order to see this, first observe that every $\psi(I_i)$ is an order filter in the root poset of $\Phi(S)$ and the $\psi(I_i)$ form a (multi)chain under inclusion. For all $i,j\in\{0,1,\ldots,k\}$,
   \begin{multline*}
    (\psi(I_i)+\psi(I_j))\cap\Phi^+(S)=((I_i\cup I(\alpha))+(I_j\cup I(\alpha))\cap\Phi^+(S)\\
    \subseteq((I_i+I_j)\cap\Phi^+(S))\cup I(\alpha)\subseteq I_{i+j}\cup I(\alpha)=\psi(I_{i+j})\text{.}
   \end{multline*}
   Also $\psi(J_i)=J_i$ for all $i\in\{0,1,\ldots,k\}$, so 
   $$(\psi(J_i)+\psi(J_j))\cap\Phi^+(S)=(\psi(J_i)+\psi(J_j))\cap\Phi^+(S\backslash\{\alpha\})\subseteq\psi(J_{i+j})$$
   for all $i,j$ with $i+j\leq k$.
   So $\Theta(I)$ is a geometric (multi)chain of order fiters in the root order of $\Phi(S)$, and the claim follows.\\
\\
 Now $\Theta$ and $\Psi$ are inverse to each other, so $\Theta$ is a bijection, as required.\\
\\
 We claim that for $\beta\in\Phi^+$, $\beta$ is a rank $l$ indecomposable element in $\Theta(I)$ if and only if $\beta$ is a rank $l$ indecomposable element in $I$ and $\beta\neq\alpha$.\\
\\
   In order to see this, first notice that for $\beta\in\theta(I_l)$, $k_{\beta}(\Theta(I))=k_{\beta}(I)$.
   The only element in $I_l\backslash\theta(I_l)=I(\alpha)$ that can be indecomposable of rank $l$ is $\alpha$, since all other elements are not minimal, so not indecomposable by \textbf{Lemma \ref{ind}}.
   So if $\beta\neq\alpha$ is a rank $l$ indecomposable element of $V$, then $\beta\in\theta(I_l)$.
   If $\beta$ were not indecomposable in $\Theta(I)$, then either $\beta=\gamma+\delta$ for $\gamma\in\theta(I_i)$, $\delta\in\theta(I_j)$, with $i+j=l$, in contradiction to $\beta$ being indecomposable in $I$, or there is a $\gamma\notin\theta(I_{t-l})$ with $\beta+\gamma\in\theta(I_t)$ and $k_{\beta+\gamma}(\Theta(I))=t$, for some $l\leq t\leq k$, also in contradiction to $\beta$ being indecomposable in $I$.
   So $\beta$ is rank $l$ indecomposable in $\Theta(I)$.\\
   \\
   Now for $\beta$ a rank $l$ indecomposable element of $\Theta(I)$, suppose for contradiction that $\beta$ were not indecomposable in $I$.
   If $\beta=\gamma+\delta$ for $\gamma\in\ I_i$, $\delta\in\ I_j$, with $i+j=l$, then $\alpha$ is not in the support of either $\gamma$ or $\delta$, so $\gamma\in\theta(I_i)$ and $\delta\in\theta(I_j)$, a contradiction to $\beta$ being indecomposable in $\Theta(I)$.
   If $\beta+\gamma\in I_t$ and $k_{\beta+\gamma}(I)=t$ for some $l\leq t\leq k$, and $\gamma\in\Phi(S\backslash\{\alpha\})$, then $\gamma\in \theta(I_t)\subseteq I_t$, as $\beta$ is indecomposable in $\Theta(I)$.
   If $\beta+\gamma\in I_t $ for some $l\leq t\leq k$ and $\gamma\notin\Phi(S\backslash\{\alpha\})$, then $\gamma\in I(\alpha)$, so $\gamma\in I_k\subseteq I_{t-l}$.
   So $\beta$ is indecomposable in $I$.
   This establishes the claim.\\
   \\
   Thus $\Theta$ is a bijection having the desired properties.

 \end{proof}

\section{Proof of the Main Result}
To prove \textbf{Theorem \ref{H=F}}, we show that the derivatives with respect to $y$ of both sides of the equation agree, as well as their specialisations at $y=1$. To do this, we need the following lemmas.
\begin{lemma}\label{y=1}
If $\Phi$ is a crystallographic root system of rank $n$, then
 $$H^k_{\Phi}(x,1)=(x-1)^nF^k_{\Phi}\left(\frac{1}{x-1},\frac{1}{x-1}\right)\text{.}$$
\begin{proof}
We have
$$(x-1)^nF^k_{\Phi}\left(\frac{1}{x-1},\frac{1}{x-1}\right)=\sum_{l,m}f_{l,m}(x-1)^{n-(l+m)}=\sum_{i=0}^nh_ix^{n-i}\text{,}$$
where $(h_0,h_1,\ldots,h_n)$ is the $h$-vector of $\Delta^{(k)}(\Phi)$. So 
$$[x^i](x-1)^nF^k_{\Phi}\left(\frac{1}{x-1},\frac{1}{x-1}\right)=h_{n-i}=Nar^{(k)}(\Phi,i)\text{,}$$
by \cite[Theorem 3.2.]{fomin05generalized}.
But
\begin{equation*}
 [x^i]H^k_{\Phi}(x,1)=Nar^{(k)}(\Phi,i)\text{,}
\end{equation*}
by \textbf{Theorem \ref{Nar}}.
 \end{proof}
\end{lemma}

\begin{lemma}[\protect{\cite[Proposition F (2)]{krattenthaler06ftriangle}}]\label{DF}
If $\Phi$ is a crystallographic root system of rank $n$, then
 $$\frac{\partial}{\partial y}F^k_{\Phi(S)}(x,y)=\sum_{\alpha\in S}F^k_{\Phi(S\backslash\{\alpha\})}(x,y)\text{.}$$
 \begin{proof}
  As mentioned in \cite{krattenthaler06ftriangle}, this can be proven in the same way as the $k=1$ case, where it is due to Chapoton \cite[Proposition 3]{chapoton04enumerative}. For completeness, as well as to highlight the analogy to the proof of \textbf{Lemma \ref{DH}}, we give the proof here.\\
  \\
  We wish to show that
  $$mf_{l,m}^k(\Phi)=\sum_{\alpha\in S}f_{l,m-1}^k(\Phi(S\backslash\{\alpha\}))\text{,}$$
  that is, we seek a bijection $\varphi$ from the set of pairs $(A,-\alpha)$ with $A\in\Delta^{(k)}(\Phi(S))$ and $-\alpha\in A\cap(-S)$ to $\amalg_{\alpha\in S}\Delta^{(k)}(\Phi(S\backslash\{\alpha\}))$ such that $\varphi(A)$ contains the same number of coloured positive roots as $A$, but exactly one less uncoloured negative simple root. By \cite[Proposition 3.5.]{fomin05generalized}, the map $\varphi(A,-\alpha)=(A\backslash\{-\alpha\},\alpha)$ is such a bijection.
 \end{proof}

\end{lemma}

\begin{lemma}\label{DH}
If $\Phi$ is a crystallographic root system of rank $n$, then
 $$\frac{\partial}{\partial y}H^k_{\Phi(S)}(x,y)=x\sum_{\alpha\in S}H^k_{\Phi(S\backslash\{\alpha\})}(x,y)\text{.}$$
 \begin{proof}
  Analogously to \textbf{Lemma \ref{DF}}, we seek a bijection $\Theta$ from the set of pairs $(I,\alpha)$ with $I\in NN^{(k)}(\Phi(S))$ and $\alpha\in I_k\cap S$ to $\amalg_{\alpha\in S}NN^{(k)}(\Phi(S\backslash\{\alpha\}))$ such that $\Theta(I)$ has exactly one less rank $k$ simple element and exactly one less rank $k$ indecomposable element than $I$.
  Such a bijection is given in \textbf{Theorem \ref{Bij}}.
 \end{proof} 
\end{lemma}
We are now in a position to prove \textbf{Theorem \ref{H=F}}.
\begin{proof}[Proof of Theorem \ref{H=F}]
 We proceed by induction on $n$. If $n=0$, both sides are equal to 1, so the result holds. If $n>0$,
 $$\frac{\partial}{\partial y}H^k_{\Phi(S)}(x,y)=x\sum_{\alpha\in S}H^k_{\Phi(S\backslash\{\alpha\})}(x,y)\text{,}$$
 by \textbf{Lemma \ref{DH}}. By induction hypothesis, this is further equal to 
 $$x\sum_{\alpha\in S}(x-1)^{n-1}F^k_{\Phi(S\backslash\{\alpha\})}\left(\frac{1}{x-1},\frac{1+(y-1)x}{x-1}\right)\text{,}$$
 which equals
 $$\frac{\partial}{\partial y}(x-1)^nF^k_{\Phi(S)}\left(\frac{1}{x-1},\frac{1+(y-1)x}{x-1}\right)$$
 by \textbf{Lemma \ref{DF}}.
 But
 $$H^k_{\Phi}(x,1)=(x-1)^nF^k_{\Phi}\left(\frac{1}{x-1},\frac{1}{x-1}\right)$$
 by \textbf{Lemma \ref{y=1}}, so
$$H^k_{\Phi}(x,y)=(x-1)^nF^k_{\Phi}\left(\frac{1}{x-1},\frac{1+(y-1)x}{x-1}\right)\text{,}$$
since the derivatives with respect to $y$ as well as the specialisations at $y=1$ of both sides agree.
\end{proof}
\section{Corollaries of the Main Result}
Specialising \textbf{Theorem \ref{H=F}} to $k=1$, we can now prove Chapoton's original conjecture.

\begin{corollary}[\protect{\cite[Conjecture 6.1.]{chapoton06sur}}]\label{k=1}
 If $\Phi$ is a crystallographic root system of rank $n$, then
 $$H^1_{\Phi}(x,y)=(1-x)^nF^1_{\Phi}\left(\frac{x}{1-x},\frac{xy}{1-x}\right)\text{.}$$
 \begin{proof}
  We have
  \begin{align} \label{H=F,k=1}
   H^1_{\Phi}(x,y)=(x-1)^nF^1_{\Phi}\left(\frac{1}{x-1},\frac{1+(y-1)x}{x-1}\right)\text{.}
  \end{align}
  But we also have \cite[Proposition 5]{chapoton04enumerative}
  \begin{align} \label{Fdual}
   F^1_{\Phi}(x,y)=(-1)^nF^1_{\Phi}(-1-x,-1-y)\text{.}
  \end{align}
  Substituting (\ref{Fdual}) into (\ref{H=F,k=1}), we obtain
  $$H^1_{\Phi}(x,y)=(1-x)^nF^1_{\Phi}\left(\frac{x}{1-x},\frac{xy}{1-x}\right)\text{.}$$
 \end{proof}

\end{corollary}

Using the $M=F$ (ex-)conjecture, we can also relate the $H$-triangle to the $M$-triangle.

\begin{corollary}[\protect{\cite[Conjecture 5.3.2.]{armstrong09generalized}}]\label{H=M}
 If $\Phi$ is a crystallographic root system of rank $n$, then
 $$H^k_{\Phi}(x,y)=(1+(y-1)x)^nM^k_{\Phi}\left(\frac{y}{y-1},\frac{(y-1)x}{1+(y-1)x}\right)\text{.}$$
 \begin{proof}
  We have
  \begin{align}\label{H=Feq}
   H^k_{\Phi}(x,y)=(x-1)^nF^k_{\Phi}\left(\frac{1}{x-1},\frac{1+(y-1)x}{x-1}\right)\text{.}
  \end{align}
  But we also have \cite[Conjecture FM]{krattenthaler06ftriangle} \cite[Theorem 1.2.]{tzanaki08faces}
  \begin{align}\label{H=Meq}
   F^k_{\Phi}(x,y)=y^nM^k_{\Phi}\left(\frac{1+y}{y-x},\frac{y-x}{y}\right)\text{.}
  \end{align}
  Substituting (\ref{H=Meq}) into (\ref{H=Feq}), we obtain 
  $$H^k_{\Phi}(x,y)=(1+(y-1)x)^nM^k_{\Phi}\left(\frac{y}{y-1},\frac{(y-1)x}{1+(y-1)x}\right)\text{.}$$
 \end{proof}
\end{corollary}

The coefficients of $F^k_{\Phi}(x,y)$ are known to be polynomials in $k$ \cite{krattenthaler06ftriangle}, so the coefficients of $H^k_{\Phi}(x,y)$ are also polynomials in $k$. Thus it makes sense to consider $H^k_{\Phi}(x,y)$ even if $k$ is not a positive integer. We can use \textbf{Corollary \ref{H=M}} to transfer a remarkable instance of combinatorial reciprocity observed by Krattenthaler \cite[Theorem 8]{krattenthaler06mtriangle} for the $M$-triangle to the $H$-triangle.

\begin{corollary}\label{recip}
 If $\Phi$ is a crystallographic root system of rank $n$, then
 $$H^k_{\Phi}(x,y)=(-1)^nH^{-k}_{\Phi}\left(1-x,\frac{-xy}{1-x}\right)\text{.}$$
 \begin{proof}
   We have
  \begin{align}\label{H=Meq2}
   H^k_{\Phi}(x,y)=(1+(y-1)x)^nM^k_{\Phi}\left(\frac{y}{y-1},\frac{(y-1)x}{1+(y-1)x}\right)\text{.}
  \end{align}
  But we also have \cite[Theorem 8]{krattenthaler06mtriangle} \cite[Theorem 1.2.]{tzanaki08faces}
  \begin{align}\label{Mrecip}
   M^k_{\Phi}(x,y)=y^nM^{-k}_{\Phi}\left(xy,\frac{1}{y}\right)\text{.}
  \end{align}
  Substituting (\ref{Mrecip}) into (\ref{H=Meq2}), we obtain 
  \begin{align}\label{H=M-}
   H^k_{\Phi}(x,y)=((y-1)x)^nM^{-k}_{\Phi}\left(\frac{xy}{1+(y-1)x},\frac{1+(y-1)x}{(y-1)x}\right)\text{.}
  \end{align}
  Inverting (\ref{H=Meq2}), we get
  \begin{align}\label{M=H}
   M^k_{\Phi}(x,y)=(1-y)^nH^k_{\Phi}\left(\frac{y(x-1)}{1-y},\frac{x}{x-1}\right)\text{.}
  \end{align}
  Substituting (\ref{M=H}) into (\ref{H=M-}), we obtain 
  $$H^k_{\Phi}(x,y)=(-1)^nH^{-k}_{\Phi}\left(1-x,\frac{-xy}{1-x}\right)\text{.}$$
 \end{proof}

\end{corollary}

For $k=1$, we can transfer a duality for the $F$-triangle to the $H$-triangle.

\begin{corollary}\label{dual}
 $$H^1_{\Phi}(x,y)=x^nH^1_{\Phi}\left(\frac{1}{x},1+(y-1)x\right)\text{.}$$
 \begin{proof}
  Inverting \textbf{Theorem \ref{H=F}}, we get
  \begin{align}\label{F=H}
   F^1_{\Phi}(x,y)=x^nH^1_{\Phi}\left(\frac{x+1}{x},\frac{y+1}{x+1}\right)\text{.}
  \end{align}
  Thus
  \begin{align*}
   H^1_{\Phi}(x,y)&=(1-x)^nF^1_{\Phi}\left(\frac{x}{1-x},\frac{xy}{1-x}\right)\\
   &=x^nH^1_{\Phi}\left(\frac{1}{x},1+(y-1)x\right)\text{,}
  \end{align*}
using \textbf{Corollary \ref{k=1}} and (\ref{F=H}).
 \end{proof}

\end{corollary}
For a bounded dominant region $R$ in the $k$-Catalan arrangement, let $CL_k(R)$ be the number of $k$-coloured ceilings of $R$.
Let $h_i^+(\Phi)$ be the number of bounded dominant regions $R$ that have exactly $n-i$ $k$-coloured ceilings. Let $h_i(\Delta_+^k(\Phi))$ be the $i$-th entry of the $h$-vector of the positive part of $\Delta^{(k)}(\Phi)$.
 
\begin{corollary}[\protect{\cite[Conjecture 1.2.]{athanasiadis06cluster}}]\label{pos}
If $\Phi$ is a crystallographic root system,
 $h_i^+(\Phi)=h_i(\Delta_+^k(\Phi))$
 for all $i$.
 \begin{proof}
  This follows from \textbf{Theorem \ref{Nar}} and \cite[Corollary 5.4.]{athanasiadis06cluster}.
 \end{proof}
\end{corollary}

\begin{corollary}\label{ceil}
 If $\Phi$ is a crystallographic root system, then
 $$H^k_{\Phi}\left(x,1-\frac{1}{x}\right)=\sum_Rx^{|CL_k(R)|}\text{,}$$
 where the sum is over all bounded dominant regions $R$ in the $k$-Catalan arrangement of $\Phi$.
 \begin{proof}
  \begin{align*}
   \sum_Rx^{|CL_k(R)|}&=\sum_{i=0}^nh_i^+(\Phi)x^{n-i}\\
   &=\sum_{i=0}^nh_i(\Delta_+^k(\Phi))x^{n-i}\\
   &=\sum_{l=0}^nf_{l,0}^k(\Phi)(x-1)^{n-l}\\
   &=(x-1)^nF^k_{\Phi}\left(\frac{1}{x-1},0\right)\\
   &=H^k_{\Phi}\left(x,1-\frac{1}{x}\right)\text{,}
  \end{align*}
using \textbf{Corollary \ref{pos}} and \textbf{Theorem \ref{H=F}}.
  
 \end{proof}

\end{corollary}

\begin{corollary}
 If $\Phi$ is a crystallographic root system of rank $n$, then
 $$\sum_Rx^{|CL_1(R)|}=x^nH^1_{\Phi}\left(\frac{1}{x},0\right)\text{,}$$
 where the sum is over all bounded dominant regions $R$ in the $1$-Catalan arrangement of $\Phi$.
 \begin{proof}
  \begin{align*}
   \sum_Rx^{|CL_1(R)|}&=H^1_{\Phi}\left(x,1-\frac{1}{x}\right)\\
   &=x^nH^1_{\Phi}\left(\frac{1}{x},0\right)\text{,}
  \end{align*}
  using \textbf{Corollary \ref{ceil}} and \textbf{Corollary \ref{dual}}.
  
 \end{proof}

\end{corollary}

\bibliographystyle{alpha}
\bibliography{literature}

\end{document}